\documentclass[11pt]{article}
\usepackage{amsthm}
\usepackage{amsmath,amsthm,amsfonts,amssymb,mathrsfs,bm,graphicx,stmaryrd,dsfont, enumerate}
\usepackage{dsfont}
\usepackage[colorlinks=true,linkcolor=blue]{hyperref}
\hypersetup{
   citecolor=blue,
   linkcolor=blue,
}
\usepackage[usenames,dvipsnames]{color}
\usepackage{fullpage}
\usepackage[american]{babel}
\usepackage[varg]{pxfonts}
\usepackage{prettyref}
\usepackage{microtype}
\usepackage{tikz}
\usepackage{pgfplots}
\pgfplotsset{compat=newest}
\usepackage[font=small,labelfont=bf]{caption}

\newtheorem{theorem}{Theorem}[section]

\newtheorem{lemma}[theorem]{Lemma}
\newtheorem{remark}[theorem]{Remark}
\newtheorem{proposition}[theorem]{Proposition}

\newtheorem{conjecture}[theorem]{Conjecture}

\newtheorem{definition}[theorem]{Definition}

\newcommand{\C}{\mathbb{C}}

\renewcommand{\P}{\mathbb{P}}

\newcommand{\vol}{\mathrm{vol}}
\newcommand{\dE}{\mathbb{E}}
\newcommand{\dP}{\mathbb{P}}
\newcommand{\eps}{\epsilon}

\newcommand\underrel[2]{\mathrel{\mathop{#2}\limits_{#1}}}

\title{Gaussian Regularization of the Pseudospectrum and Davies' Conjecture}
\author{Jess Banks \\ UC Berkeley \and Archit Kulkarni \\ UC Berkeley \and Satyaki Mukherjee \\ UC Berkeley \and Nikhil Srivastava \\{UC Berkeley}}
\date{\today}

\begin{document}
\maketitle
\begin{abstract} 
	A matrix $A\in\C^{n\times n}$ is diagonalizable if it has a basis of linearly independent eigenvectors. 
	Since the set of nondiagonalizable matrices has measure zero, every $A\in \C^{n\times n}$ is the limit of diagonalizable matrices. 
	We prove a quantitative version of this fact conjectured by E.B. Davies: for each $\delta\in (0,1)$, every matrix $A\in \C^{n\times n}$ is at least $\delta\|A\|$-close to one whose eigenvectors have condition number at worst $c_n/\delta$, for some $c_n$ depending only on $n$. We further show that the dependence on $\delta$ cannot
	be improved to $1/\delta^p$ for any constant $p<1$. 
	
	Our proof uses tools from random matrix theory to show that the pseudospectrum of $A$ can be regularized with the addition of a complex Gaussian perturbation.  Along the way, we explain how a variant of a theorem of \'Sniady implies a conjecture of Sankar, Spielman and Teng on the optimal constant for smoothed analysis of condition numbers.
	
 \end{abstract}

\section{Introduction}
A matrix $A\in\C^{n\times n}$ is diagonalizable if it can be written as
$ A = VDV^{-1},$
where $D$ is diagonal and $V$ is a matrix  consisting of linearly independent eigenvectors of $A$. 
Further, $A$ is normal if and only if $V^{-1}=V^*$, or in other words if the eigenvectors can be chosen to be orthogonal.
A fundamental quantity capturing the nonnormality of a matrix is the {\em eigenvector condition number}
$$ 
	\kappa_V(A) \triangleq \inf_{V:A=VDV^{-1}} \|V\|\|V^{-1}\|,
$$
which ranges between $1$ and $\infty$ when $A$ is normal and nondiagonalizable respectively, where $\|\cdot\|$  denotes the operator norm. 
Matrices with small $\kappa_V$ enjoy many
of the desirable properties of normal matrices, such as stability of the spectrum under small perturbations (this is the content of the Bauer-Fike theorem \cite{bauer1960norms}).
In this paper we study the following question posed by E. B. Davies in \cite{davies}:
\begin{quote}
	{\em How well can an arbitrary matrix be approximated by one with a small eigenvector condition number?}
\end{quote}
Our main theorem is as follows.

\begin{theorem}\label{thm:a} 
	Suppose $A\in\C^{n\times n}$ and $\delta\in (0,1)$. Then there is a matrix $E\in\C^{n\times n}$ such that $\|E\|\le \delta\|A\|$ and
	$$ 
		\kappa_V(A+E)\le 4n^{3/2}\left(1+\frac1\delta\right).
	$$
\end{theorem}

In other words, every matrix is at most inverse polynomially close to a matrix whose eigenvectors have condition number at most polynomial in the dimension. The previously best known general bound in such a result was \cite[Theorem 3.8]{davies}:
\begin{equation}\label{eqn:oldbound}\kappa_V(A+E)\le \left(\frac{n}{\delta}\right)^{(n-1)/2},\end{equation}
so our theorem constitutes an exponential improvement in the dependence on both $\delta$ and $n$. We show in Proposition \ref{prop:lowerdelta} that the
$1/\delta$-dependence in Theorem
\ref{thm:a} cannot be improved beyond $1/\delta^{1-1/n}$, so our bound is essentially optimal in $\delta$ for large $n$.

\subsection{Davies' Conjecture}\label{sec:davies}
Theorem \ref{thm:a} implies a positive resolution to a conjecture of Davies \cite{davies}.
\begin{conjecture} \label{conj:davies} For every positive integer $n$ there is a constant $c_n$ such that for every $A\in\C^{n\times n}$ with $\|A\|\le 1$ and $\epsilon \in (0, 1)$:
\begin{equation} \label{eqn:conj}
	\inf_{E\in\C^{n\times n}} \left(\kappa_V(A+E)\epsilon + \|E\|\right) \le c_n\sqrt{\epsilon}.
\end{equation}
	\begin{proof}[Proof of Conjecture \ref{conj:davies}] Given $\epsilon>0$,
	set $\delta = d_n \sqrt{\eps}$ for some $d_n > 0$ and apply Theorem \ref{thm:a}.  This yields $c_n = 4n^{3/2} + 4n^{3/2}/d_n + d_n$.  This is minimized at $d_n = 2n^{3/4}$, which yields $c_n = 4n^{3/2} + 4n^{3/4} \le 8n^{3/2}.$
	\end{proof}
\end{conjecture}
The phrasing of Conjecture \ref{conj:davies} is motivated by a particular application in numerical analysis. 
Suppose one wants to evaluate analytic functions $f(A)$ of a given matrix $A$, which may be nonnormal. 
If $A$ is diagonalizable, one can use the formula $f(A)=Vf(D)V^{-1}$, where $f(D)$ means the function is applied to the scalar diagonal entries of $D$. 
However, this may be numerically infeasible if $\kappa_V(A)$ is very large: if all computations are carried to precision $\epsilon$, the result may be off by an error of $\kappa_V(A)\epsilon$.
Davies' idea was to replace $A$ by a perturbation $A+E$ with a much smaller $\kappa_V(A+E)$, and compute $f(A+E)$ instead. In \cite[Theorem 2.4]{davies}, he showed that the net error incurred by this scheme for a given $\epsilon>0$ and sufficiently regular $f$ is controlled by:
$$ \kappa_V(A+E)\epsilon + \|E\|,$$
which is the quantity appearing in \eqref{eqn:conj}. The key desirable feature of \eqref{eqn:conj} is the dimension-independent fractional power of $\epsilon$ on the right-hand side, which shows that the total error scales slowly. 

Davies proved his conjecture in the special case of upper triangular Toeplitz matrices, in dimension $n= 3$ with the constant $c_n = 2$, as well as in the general case with the weaker dimension-dependent and nonconstructive bound $(n+1)\epsilon^{2/(n+1)}$. This last result corresponds to \eqref{eqn:oldbound} above.  
He also speculated that a {\em random}
regularizing perturbation $E$ suffices to prove Conjecture \ref{conj:davies},
and presented empirical evidence to that effect.  Our proof of Theorem
\ref{thm:a} below indeed follows this strategy.

\subsection{Gaussian Regularization}
Theorem \ref{thm:a} follows from a probabilistic result concerning complex Gaussian perturbations of a given matrix $A$.
To state our result, we recall two standard notions.

\begin{definition} 
	A {\em complex Ginibre matrix} is an $n\times n$ random matrix $G_n=(g_{ij})$ with i.i.d complex Gaussian entries $g_{ij}\sim N(0,1_\C/n)$, by which we mean $\dE g_{ij}=0$ and $\dE |g_{ij}|^2=1/n$. Equivalently the real and imaginary parts of each $g_{ij}$ are independent $N(0,1/2n)$ random variables. 
\end{definition}

\begin{definition} 
	Let $M\in\C^{n\times n}$ have distinct eigenvalues $\lambda_1,\ldots,\lambda_n$, and spectral expansion
	$$
		M = \sum_{i=1}^n \lambda_i v_i w_i^\ast = VDV^{-1},
	$$
	where the right and left eigenvectors $v_i$ and $w_i^\ast$ are the columns and rows of $V$ and $V^{-1}$, respectively, normalized so that $w_i^*v_i=1$. 
	The \emph{eigenvalue condition number of $\lambda_i$} is defined as:
	$$
	    \kappa(\lambda_i) \triangleq \left\| {v_i w_i^\ast}\right\| =\|v_i\|\|w_i\|.
	$$
\end{definition}
The $\kappa(\lambda_i)$'s are called condition numbers because they determine the sensitivity of the $\lambda_i$ to perturbations of the matrix.
We show that adding a small Ginibre perturbation regularizes the eigenvalue condition numbers of any matrix in the following averaged sense.
\begin{theorem} \label{thm:b} 
	Suppose $A\in\C^{n\times n}$ with $\|A\|\le 1$ and $\delta\in (0,1)$. Let $G_n$ be a complex Ginibre matrix, and let $\lambda_1,\ldots,\lambda_n\in \C$ be the (random) eigenvalues of $A+\delta G_n$. Then for every measurable open set $B\subset \C,$
	$$ 
		\dE \sum_{\lambda_i \in B} \kappa(\lambda_i)^2 \le \frac{n^2}{\pi \delta^2}\vol(B).
	$$
\end{theorem}
\noindent Note that the $\kappa(\lambda_i)$ appearing above are well-defined because $A+\delta G_n$ has distinct eigenvalues with probability one. 

\subsection{Related Work} 
{\em Random Matrix Theory.} There have been numerous studies of the eigenvalue condition numbers $\kappa(\lambda_i)^2$, sometimes called eigenvector \emph{overlaps} in the random matrix theory and mathematical physics literature, for non-Hermitian random matrix models of type $A+\delta G_n$.
In the centered case $A=0$ and $\delta=1$ of a standard complex Ginibre matrix, 
the seminal work of Chalker and Mehlig \cite{chalker1998eigenvector} calculated the large-$n$ limit of
the conditional expectations $$\dE [\kappa(\lambda)^2|
\lambda=z]\underrel{n\rightarrow\infty}{\sim} n(1-|z|^2),$$ whenever $|z|<1$.
Recent works by Bourgade and Dubach \cite{bourgade2018distribution} and Fyodorov \cite{fyodorov2018statistics}
improved on this substantially by giving exact nonasymptotic formulas for the distribution of $\kappa(\lambda)^2$ conditional on the location of the eigenvalue $\lambda$, as well as
concise descriptions of the scaling limits for these formulas.
The paper \cite{benaych2018eigenvectors} proved (in the more general setup of invariant ensembles) that the angles between the right
eigenvectors $(v_i^*v_j)/\|v_i\|\|v_j\|$ have subgaussian tails, which has some bearing on $\kappa_V$ (for instance, a small angle between unit eigenvectors causes $\Vert V^{-1} \Vert$ and therefore $\kappa_V$ to blow up.)

In the non-centered case, Davies and Hager \cite{davies2009perturbations}
showed that if $A$ is a Jordan block and $\delta=n^{-\alpha}$ for some
appropriate $\alpha$, then almost all of the eigenvalues of $A+\delta G_n$ lie
near a circle of radius $\delta^{1/n}$ with probability $1-o_n(1)$.  Basak,
Paquette, and Zeitouni \cite{basak2018spectrum, basak2019regularization} showed
that for a sequence of banded Toeplitz matrices $A_n$ with a finite symbol, the
spectral measures of $A_n+n^{-\alpha}G_n$ converge weakly in probability, as $n\to \infty$, to a
predictable density determined by the symbol. Both of the above results were
recently and substantially improved by Sj\"ostrand and Vogel
\cite{sjoestrand2019general, sjoestrand2019toeplitz} who proved that for any
Toeplitz $A$, almost all of the eigenvalues of $A+n^{-\alpha}G_n$ are close to the
symbol curve of $A$ with exponentially good probability in $n$. Note that none of the results
mentioned in this paragraph explicitly discuss the $\kappa(\lambda_i)$; however, they do deal qualitatively
with related phenomena surrounding spectral instability of non-Hermitian matrices.

The idea of managing spectral instability by adding a random perturbation can
be traced back to the influential papers of Haagerup and Larsen
\cite{haagerup2000brown} and \'Sniady \cite{sniady2002random} (see also
\cite{guionnet2014convergence,noy2014regularization}), who used it to study
convergence of the eigenvalues of certain non-Hermitian random matrices to a limiting Brown
measure, in the context of free probability theory.

There are three notable differences between Theorem \ref{thm:b} and the results mentioned above:
\begin{enumerate}
	\item Our result is much coarser, and only guarantees an upper bound on the $\dE \kappa(\lambda_i)^2$, 
		rather than a precise description of any distribution, limiting or not.
	\item It applies to any $A\in\C^{n\times n}$ and $\delta\in (0,1)$.
	\item It is completely nonasymptotic and does not require $n\rightarrow\infty$ or even sufficiently large $n$.
\end{enumerate}

{\em Numerical Analysis.} In the numerical linear algebra literature, several works have analyzed the condition numbers of Gaussian matrices (notably the seminal results of Demmel \cite{demmel1983numerical} and Edelman \cite{edelman1988eigenvalues}) as well as perturbations of arbitrary matrices by Gaussian matrices (beginning with \cite{sankar2006smoothed}) in the nonasymptotic regime. In contrast, this paper studies the condition numbers of the {\em eigenvectors} of such matrices, rather than of the matrices themselves.

The idea of approximating matrix functions by adding a regularizing perturbation was introduced in \cite{davies} and has since appeared in several works regarding numerical computation of the matrix logarithm, sine, cosine, and related functions \cite{al2013computing,higham2013improved,al2015new,nadukandi2018computing,defez2019efficient}.

\subsection{Techniques and Organization}
The proofs of Theorems \ref{thm:a} and \ref{thm:b} are quite simple and rely on an interplay between various notions of spectral stability.  In addition to $\kappa_V$ and the $\kappa(\lambda_i)$, we
will heavily use the notion of the {\em $\epsilon-$pseudospectrum} of a matrix $M$, defined for $\epsilon>0$ as:
\begin{align}
    \Lambda_\epsilon(M) 
	&\triangleq \left\{z \in \mathbb{C} : z \in \Lambda(M + E) \text{ for some $\|E\| < \epsilon$}\right\} \\
	&= \left\{z \in \mathbb{C} : \|(z{I} - M)^{-1}\| > 1/\epsilon \right\} \label{eqn:ps1}\\
	&= \left\{z \in \mathbb{C} :  \sigma_n(z{I}-M)< \epsilon \right\} \label{eqn:ps2},
\end{align}
where $\Lambda(M)$ denotes the spectrum $M$.
For a proof of the equivalence of these three sets and a comprehensive treatment of pseudospectra, see the beautiful book
by Trefethen and Embree \cite{trefethen2005spectra}.
Note that for a normal matrix, we have 
$$
	\Lambda_\epsilon(M) = \Lambda(M)+ \bigcup_{i}D(\lambda_i, \epsilon),
$$
whereas for a nonnormal matrix such as a Jordan block, $\Lambda_\epsilon$ can be much larger.  Figure 1 illustrates the regularizing effect of a small complex Gaussian perturbation on the pseudospectrum of a nondiagonalizable matrix, which is the key phenomenon underlying our results. 
\begin{figure}
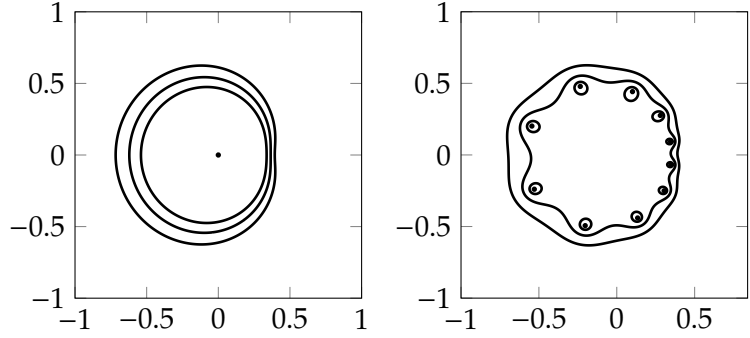
\label{fig:toeplitz}
\input{Toeplitz.tex}
\input{Toeplitz_perturbed.tex}
\centering
\caption{$T$ is a sample of an upper triangular $10\times 10$ Toeplitz matrix with zeros on the diagonal and independent (modulo the Toeplitz structure) standard real Gaussian entries above the diagonal.  
Pictured is the boundary of the $\eps$-pseudospectrum of $T$ (left) and $T+10^{-6} G$ (right) for $\eps= 10^{-5}$, $\eps = 10^{-5.5}$, and $\eps = 10^{-6}$, along with the spectra.  These plots were generated with the MATLAB package EigTool \cite{wright2002eigtool}.} 
\end{figure}


 To analyze this phenomenon, we first collect some tools from random matrix theory in Section \ref{sec:rmt}, along the way proving a conjecture of Sankar, Spielman, and Teng \cite{sankar2006smoothed} regarding the optimal constant in their smoothed analysis of condition numbers of matrices under {\em real} Gaussian perturbations in Section \ref{sec:st}. Section \ref{sec:upper} contains the proofs of our main results, Theorems \ref{thm:a} and \ref{thm:b}.
In Section \ref{sec:lower}, we prove optimality of the $1/\delta$-dependence in Theorem
\ref{thm:a} as discussed above, and show that Theorem \ref{thm:b} is sharp up to a small constant factor. We conclude with a discussion of some open problems in Section \ref{sec:conc}.

\subsubsection*{Notation} We denote the singular values of an $n\times n$
matrix by $\sigma_1(M)\ge \ldots\ge \sigma_n(M)$, its operator and Frobenius
(Hilbert-Schmidt) norms by $\|M\|$ and $\|M\|_F$, and its condition number by
$\kappa(M)\triangleq\sigma_1(M)/\sigma_n(M)$. Open disks in the complex plane will be written as $D(z_0,r)\triangleq \{z\in \C: |z-z_0|<r\}.$
We will often write $G$ for a standard complex Gaussian matrix with $N(0,1_\C)$ entries, and $G_n = n^{-1/2}G$ for a (normalized) Ginibre matrix.

\section{Tools from Random Matrix Theory}\label{sec:rmt}

\subsection{Nonasymptotic Extreme Singular Value Estimates}

Let us record some standard non-asymptotic estimates for the extreme singular values of complex Ginibre matrices. The lower tail behavior of the smallest singular value of a Ginibre matrix was worked out by Edelman in the unnormalized scaling of i.i.d. $N(0,1_\C$) entries \cite[Chapter 5]{edelman1988eigenvalues}, and in our setting it translates to:

\begin{theorem} \label{thm:edelman} 
	 For a complex Ginibre matrix $G_n$, 
	$$\dP[\sigma_n(G_n) < \epsilon] = 1 - e^{-\epsilon^2n^2} \le \epsilon^2 n^2.$$
\end{theorem}    

We will also require a cruder tail estimate on the largest singular value. We believe the lemma holds with a constant $2$ instead of $2\sqrt{2}$, but did not find a reference to a nonasymptotic result to this effect; since the difference is not very consequential in this context, we reduce to the real case.
\begin{lemma} \label{lem:s1} For a complex Ginibre matrix $G_n$,
	$$ \dP[\sigma_1(G_n)> 2\sqrt{2}+t]\le 2\exp(-nt^2).$$
\end{lemma}
\begin{proof}
	We can write $G_n = \frac{1}{\sqrt{2}}(X+iY)$ where $X$ and $Y$ are independent with i.i.d. {\em real} $N(0,1/n)$ entries.
	It is well-known (e.g. \cite[Theorem II.11]{davidson2001local}) that:
	$$\dE \sigma_1(G_n) \le \frac{2}{\sqrt{2}} \dE \|X\| \le 2\sqrt{2}.$$
	Lipschitz concentration of functions of real Gaussian random variables yields the result.
\end{proof}

\subsection{\'Sniady's Comparison Theorem}
To bound the least singular value of noncentered Gaussian matrices, we will lean on a remarkable theorem of \'Sniady \cite{sniady2002random}.
\begin{theorem}[\'Sniady]  \label{thm:sniady}
	Let $A_1$ and $A_2$ be $n \times n$ complex matrices such that $\sigma_i(A_1) \le \sigma_i(A_2)$ for all $1 \le i \le n$.  Assume further that $\sigma_i(A_1) \ne \sigma_j(A_1)$ and $\sigma_i(A_2) \ne \sigma_j(A_2)$ for all $i \ne j$.  Then for every $t \ge 0$, there exists a joint distribution on pairs of $n \times n$ complex matrices $(G_1, G_2)$ such that 
	\begin{enumerate}[(i)]
    	\item the marginals $G_1$ and $G_2$ are distributed as (normalized) complex Ginibre matrices $G_n$, and
    	\item almost surely $\sigma_i(A_1 + \sqrt{t} G_1) \le \sigma_i(A_2 + \sqrt{t} G_2)$ for every $i$.
 	\end{enumerate} 
\end{theorem}
\noindent We will briefly sketch the proof of this theorem for the reader's benefit, since it is quite beautiful and we will need to perform a slight modification to prove the conjecture of Sankar-Spielman-Teng in the next subsection.
\begin{proof}[Sketch of proof]
	The key insight of the proof is that it is possible to couple the distributions of $G_1$ and $G_2$ through their singular values. To do so, one first derives a stochastic differential equation satisfied by the singular values $s_1,...,s_n$ of a matrix Brownian motion (i.e., a matrix whose entries are independent complex Brownian motions):
	\begin{equation} \label{sde} 
		ds_i = \frac{1}{\sqrt{2n}}dB_i + \frac{dt}{2s_i} \left(1 - \frac{1}{2n} + \sum_{j \ne i} \frac{s_i^2 + s_j^2}{n(s_i^2 - s_j^2)} \right), 
	\end{equation}
	where the $B_i$ are independent standard real Brownian motions.
	Next, one uses a single $n$-tuple of real Brownian motions $B_1,...,B_n$ to drive two processes $(s_1^{(1)}, \dots, s_n^{(1)})$ and $(s_1^{(2)}, \dots, s_n^{(2)})$ according to (\ref{sde}), with initial conditions $s_i^{(1)}(0) = \sigma_i(A_1)$ and $s_i^{(2)}(0) = \sigma_i(A_2)$ for all $i$. (To do this rigorously, one needs existence and uniqueness of strong solutions to the above SDE; this is shown in \cite{konig2001eigenvalues} under the hypothesis $s_i(0) \ne s_j(0)$ for all $i \ne j$.)

	Things have been arranged so that the joint distribution of $(s_1^{(j)}, \dots, s_n^{(j)})$ at time $t$ matches the joint distribution of the singular values of $A_j + \sqrt{t} G_j$ for each $j = 1,2$. One can then sample unitaries $U_j$ and $V_j$ from the distribution arising from the singular value decomposition $A_j+\sqrt{t}G_j = U_jD_jV_j^*$, conditioned on $D_j = \operatorname{diag}(s_1^{(j)}, \dots, s_n^{(j)})$. Thus each $G_j$ is separately Ginibre-distributed. However, $A_1 + \sqrt{t}G_1$ and $A_2 + \sqrt{t}G_2$ are coupled through the shared randomness driving the evolution of their singular values. In particular, since the same $B_i$ were used for both processes, from (\ref{sde}) one can verify that the $n$ differences $s_i^{(2)} - s_i^{(1)}$ are $C^1$ in $t$.  By taking derivatives, one can then show the desired monotonicity property: if $s_i^{(2)} - s_i^{(1)} \ge 0$ holds for all $i$ at $t=0$, it must hold for all $t \ge 0$.
\end{proof}

\subsection{Sankar-Spielman-Teng Conjecture}\label{sec:st}
The proof technique of \'Sniady can be adapted to prove a counterpart of Theorem \ref{thm:sniady} for {\em real} Ginibre perturbations (by this we mean matrices with i.i.d. real $N(0,1/n)$ entries).  Because a rigorous proof requires some stochastic analysis, we defer the proof and discussion of the following theorem to the appendix.
\begin{theorem}\label{thm:sniadyreal}
	Let $A_1$ and $A_2$ be $n \times n$ real matrices such that $\sigma_i(A_1) \le \sigma_i(A_2)$ for all $1 \le i \le n$.  Assume further that $\sigma_i(A_1) \ne \sigma_j(A_1)$ and $\sigma_i(A_2) \ne \sigma_j(A_2)$ for all $i \ne j$.  Then for every $t \ge 0$, there exists a joint distribution on pairs of real  $n \times n$ matrices $(G_1, G_2)$ such that 
	\begin{enumerate}[(i)]
    	\item the marginals $G_1$ and $G_2$ are distributed as real Ginibre matrices (with i.i.d. $N(0,1/n)$ entries), and
    	\item almost surely $\sigma_i(A_1 + \sqrt{t} G_1) \le \sigma_i(A_2 + \sqrt{t} G_2)$ for every $i$.
 	\end{enumerate} 
\end{theorem}

\noindent This resolves Conjecture 1 in \cite{sankar2006smoothed}, which we restate below as a proposition:
\begin{proposition} \label{prop:sst}
	Let $G$ be an $n \times n$ matrix with i.i.d. real $N(0,1)$ entries, and $A$ be any $n \times n$ matrix with real entries. Then 
	\[ 
		\P[\sigma_n(A + G) < \eps] \le \eps \sqrt{n}.
	\]
\end{proposition}
\begin{proof}
The case $A=0$ is a result of Edelman \cite{edelman1988eigenvalues}.  The proposition for general $A$ would then follow from Theorem \ref{thm:sniadyreal} with $A_1 = 0$ and $A_2 = A$ if not for the hypothesis $\sigma_i(A_1) \ne \sigma_j(A_1)$ and $\sigma_i(A_2) \ne \sigma_j(A_2)$ for all $i \ne j$.  So we approach 0 and $A$ by matrices satisfying this hypothesis, apply Theorem \ref{thm:sniadyreal}, and take limits, using the continuous mapping theorem and continuity of $\sigma_n(\cdot)$.
\end{proof}

\section{Proof of Theorems \ref{thm:a} and \ref{thm:b}}\label{sec:upper}
We begin with a lemma relating the eigenvector and eigenvalue condition numbers. For related results, including an extension of this lemma to the more general context of block diagonalization, see the thesis of Demmel \cite[Equation 3.6]{demmel1983numerical}.
\begin{lemma} \label{lem:condineq}
Let $M$ be an $n \times n$ matrix with distinct eigenvalues, and let $V$ be the matrix whose columns are the eigenvectors of $M$ normalized to have unit norm.  Then
\[\kappa(V)\le \sqrt{n \sum_{i=1}^n \kappa(\lambda_i)^2}.\] 
\end{lemma}
\begin{proof}
Note that the left eigenvectors $w_i$ are the rows of $V^{-1}$.  Then $\Vert V \Vert_F^2 = n$ and $\Vert V^{-1} \Vert_F^2 = \sum_{i=1}^n \Vert w_i\Vert^2 = \sum_{i=1}^n \kappa(\lambda_i)^2$, so 
\[\kappa(V)= \Vert V \Vert \Vert V^{-1} \Vert \le \Vert V \Vert_F \Vert V^{-1} \Vert_F =\sqrt{n \sum_{i=1}^n \kappa(\lambda_i)^2}.\]
\end{proof}

We can now prove the main theorem.
\begin{proof}[Proof of Theorem \ref{thm:a} given Theorem \ref{thm:b}] 
Let $\lambda_1,\ldots,\lambda_n$ be the eigenvalues of the random matrix $A+\delta G_n$, and $t > 2 \sqrt{2}$ and  $s > 1$ be parameters to be optimized later.  Davies' original bound (\ref{eqn:oldbound}) implies our bound for $n \le 3$, so assume $n \ge 4$.
 Then Lemma \ref{lem:s1} tells us that 
	\begin{equation}\dP[\Vert \delta G_n \Vert \ge t\delta] \le 2 e^{-4 (t - 2 \sqrt{2})^2}.\end{equation}
Letting $B=D(0,\|A\|+t\delta)$, we have
	\begin{equation} \label{eqn:gsmall}
		\dP \left[\sum_{\lambda_i\in B}\kappa(\lambda_i)^2 \neq \sum_{i\le n} \kappa(\lambda_i)^2\right] \le \dP[\|\delta G_n\| \ge t\delta]\le 2 e^{-4 (t - 2 \sqrt{2})^2}
	\end{equation}
	since $\max_{i\le n} |\lambda_i|\le \|A\|+\|\delta G_n\|.$
	On the other hand, by Theorem \ref{thm:b} applied to $B$ and Markov's inequality:
	\begin{equation}\label{eqn:bgood}
		\dP\left[\sum_{\lambda_i\in B} \kappa(\lambda_i)^2 \ge s \frac{n^2\vol(B)}{\delta^2\pi}\right]\le  \frac{1}{s}.
	\end{equation}
By the union bound, if we choose $s$ and $t$ such that
\begin{equation}\label{eqn:st} 
2 e^{-4 (t - 2 \sqrt{2})^2} + \frac{1}{s} < 1
\end{equation} then there exists a choice of $G_n$ such that neither of the events \eqref{eqn:gsmall}, \eqref{eqn:bgood} occurs.
	Letting $E=\delta G_n$ for this choice, we have
	$$ \sum_{i=1}^n \kappa(\lambda_i)^2 = \sum_{i\in B} \kappa(\lambda_i)^2 \le s\frac{n^2\vol(B)}{\pi\delta^2}.$$
	Taking a square root and applying Lemma \ref{lem:condineq}, we have
	$$\kappa_V(A+E) \le \frac{\sqrt{s}n^{3/2}}{\delta}(\|A\|+t\delta) \le \frac{\sqrt{s}n^{3/2}\|A\|}{\delta} + t \sqrt{s}n^{3/2}.$$
	Because $\|E\|\le t\delta$ and not $\delta$, replacing $\delta$ by $\delta/t$ yields the bound
	$$ \kappa_V(A+E)\le  \frac{t \sqrt{s} n^{3/2}\|A\|}{\delta}+t \sqrt{s} n^{3/2}.$$
	
	To get the best bound, we must minimize $t \sqrt{s}$ subject to the constraints (\ref{eqn:st}), $t > 2 \sqrt{2}$ and $s > 1$.  Solving for $s$ this becomes a univariate optimization problem, and one can check numerically that the optimum is achieved at $t \approx 3.7487$ and $t\sqrt{s} \approx 3.8822 < 4$, as advertised.
 \end{proof}

We begin the proof of Theorem \ref{thm:b} by relating the eigenvalue condition numbers of a matrix to the rate at which its pseudospectrum $\Lambda_\epsilon$ shrinks as a function of the
parameter $\epsilon>0$.  The following proposition is not new; the proof
essentially appears for example in Section 3.6 of
\cite{bourgade2018distribution}, but we include it for completeness since it is
critical to our argument.
\begin{lemma}[Limiting Area of the Pseudospectrum] \label{prop:limit}
    Let $M$ be an $n \times n$ matrix with $n$ distinct eigenvalues $\lambda_1, \dots, \lambda_n$ and let $B\subset \C$ be a measurable open set.  Then
    \[ \lim_{\eps \to 0} \frac{\vol(\Lambda_\eps(M)\cap B)}{\eps^2} = \pi \sum_{\lambda_i\in B}^n \kappa(\lambda_i)^2.\] 
\end{lemma} 

\begin{proof}
	Write the spectral decomposition 
	\[ 
		(zI-M)^{-1} = \sum_{i=1}^n \frac{v_i w_i^*}{z - \lambda_i}, 
	\]
	where the $v_i$ and $w_i^\ast$ are right and left eigenvectors of $M$,
	respectively. Since the $\lambda_i$ are distinct, we may choose 
	$\epsilon_0 > 0$ sufficiently small to guarantee that there exists a constant $C>0$ satisfying (1) the disks $D(\lambda_i,\epsilon_0)$ are disjoint; (2) for every $\lambda_i\in B$ the disk $D(\lambda_i,\epsilon_0)$ is contained in $B$; and (3) whenever $z\in D(\lambda_i,\epsilon_0)$ for some $i$,
	\begin{equation}
		\Vert (zI-M)^{-1} \Vert \ge \frac{\Vert v_i w_i^* \Vert }{|z-\lambda_i|} - C  = \frac{\kappa(\lambda_i)}{|z - \lambda_i|} - C. \label{eq:ps-pole}
	\end{equation}
	Using the definition of the $\eps$-pseudospectrum in (\ref{eqn:ps1}),
	we rearrange (\ref{eq:ps-pole}) to obtain
	\[ 
		\Lambda_\eps(M)\cap B\supset \left\{ z : |z - \lambda_i| \le \min\left\{ \epsilon_0, \frac{\kappa(\lambda_i) \eps}{1 + \eps C}\right\},\textrm{ for some }  \lambda_i\in B\right\}.
	\]
	Thus, taking $\eps$ small enough, we have
	\[ 
		\liminf_{\eps \to 0} \frac{\vol(\Lambda_\eps(M)\cap B)}{\eps^2} \ge \pi \sum_{i=1}^n\kappa(\lambda_i)^2.
	\] 

    For the opposite inequality, Theorem 52.1 of \cite{trefethen2005spectra} states that the $\eps$-pseudospectrum is contained in disks around the eigenvalues $\lambda_i$ of radii $\eps \kappa(\lambda_i) + O(\eps^2)$. Choosing $\epsilon$ small enough so that for $\lambda_i\in B$ these disks are entirely contained in $B$:
    \begin{align*} 
	    \vol(\Lambda_\eps\cap B) &\le \sum_{\lambda_i\in B} \pi (\eps \kappa(\lambda_i) + O(\eps^2))^2 
	    \Rightarrow \limsup_{\eps \to 0}  \frac{\vol(\Lambda_\eps\cap B)}{\eps^2} \le  \sum_{\lambda_i\in B} \pi \kappa(\lambda_i)^2.
    \end{align*}
\end{proof}

Next, we show that for fixed $\epsilon>0$, any particular point $z\in \C$ is
unlikely to be in $\Lambda_\epsilon(A+\delta G_n)$. This is based on the
following singular value estimate, which generalizes Theorem \ref{thm:edelman}.
\begin{lemma}[Small Ball Estimate for $\sigma_n$]\label{lem:anti}
    Let $M$ be an $n\times n$ matrix with complex entries, and $G$ be drawn from the Ginibre ensemble. Then for all $\delta>0$ and $\epsilon>0$
    $$
        \dP\left[\sigma_n(M + \delta G_n) < \epsilon\right] \le n^2 \frac{\epsilon^2}{\delta^2}.
    $$
\end{lemma}
\begin{proof}
	Repeat the proof of Proposition \ref{prop:sst} using instead Theorems \ref{thm:edelman} and \ref{thm:sniady}. 
\end{proof}
\begin{remark} \label{rem:real} 
	If one is willing to lose a small constant factor in the bound, Lemma \ref{lem:anti} has an elementary geometric proof (which avoids stochastic calculus), essentially identical to the proof of Sankar-Spielman-Teng \cite[Theorem 3.1]{sankar2006smoothed} in the case of real Ginibre perturbations. Note however that it is crucial to use a {\em complex} Gaussian for our purposes. A real Gaussian would yield a small ball estimate of order $\epsilon$ (see Proposition \ref{prop:sst}) rather than $\epsilon^2$, which is not good enough to take the limit below.
\end{remark}
\begin{proof}[Proof of Theorem \ref{thm:b}]
	For every $z\in \C$ we have the upper bound
	\begin{equation}
	\label{eqn:smoothed}
		\dP[z \in \Lambda_\epsilon(A + \delta G_n)] = \dP [\sigma_n(zI-(A+\delta G_n))<\epsilon]  \le n^2 \frac{\eps^2}{\delta^2},
	\end{equation}
		by applying Lemma \ref{lem:anti} to $M=zI-A$ and noting that $G$ and $-G$ have the same distribution.

	Fix a measurable open set $B\subset \C$. Then
	\begin{align}\label{eqn:volbound}
		\dE \,\vol (\Lambda_\epsilon(A + \delta G_n)\cap B) \nonumber
		&= \dE \int_B \mathds{1}_{\{z \in \Lambda_\epsilon(A + \delta G_n)\}}\, dz \nonumber\\
		&= \int_B \dE \{z \in \Lambda_\epsilon(A + \delta G_n)\} \,dz & &\textrm{by Fubini} \nonumber\\
		&\le \int_B n^2\frac{\eps^2}{\delta^2}\,dz & & \textrm{by \eqref{eqn:smoothed}} \nonumber\\
		&= n^2\frac{\eps^2}{\delta^2}\vol(B)
	\end{align}
	where the integrals are with respect to Lebesgue measure on $\C$. 
	Finally, taking a limit as $\epsilon\to 0$ yields the desired bound:
	\begin{align*}
		\dE \sum_{\lambda_i\in B} \kappa(\lambda_i^2) 
		&= \dE \liminf_{ \epsilon\to 0} \frac{\vol(\Lambda_\epsilon(A + \delta G_n)\cap B)}{\pi \epsilon^2} & &\textrm{by Lemma \ref{prop:limit}}\\
		&\le \liminf_{\epsilon\to 0}\dE \frac{\vol(\Lambda_\epsilon(A + \delta G_n)\cap B)}{\pi\epsilon^2} & & \textrm{by Fatou's Lemma}\\
		&\le \frac{n^2\vol(B)}{\pi \delta^2} & &\textrm{by \eqref{eqn:volbound}}.
	\end{align*}
\end{proof}

\section{Optimality of the Bounds}\label{sec:lower}
We first show that Theorem \ref{thm:a} has essentially the optimal dependence on $\delta$ for $n$ large.  The  example  which  requires  this dependence is simply a Jordan block $J$, for which Davies \cite{davies} established the upper bound $\kappa_V(J+\delta E) \le 2/\delta^{1-1/n}$, for some $E$ with $||E|| < 1$.


\begin{proposition} \label{prop:lowerdelta}
	Fix $n > 0$ and let $J \in \mathbb{C}^{n \times n}$ be the upper triangular Jordan block with ones on the superdiagonal and zeros everywhere else. Then there exist $c_n > 0$ and $\delta_n>0$ such that for all $E \in \mathbb{C}^{n \times n}$ with $\Vert E \Vert \le 1$ and all $\delta < \delta_n$, we have \[\kappa_V(J + \delta E) \ge \frac{c_n}{\delta^{1 - 1/n}}.\]
\end{proposition}

\begin{proof}
	As a warm-up, we'll need the following bound on the pseudospectrum of $J$. Let $\lambda$ be an eigenvalue of $J + \delta E$, with $v$ its associated right eigenvector; then $(J + \delta E)^nv = \lambda^n v$ and, accordingly, $|\lambda|^n \le \|(J + \delta E)^n\|$. Expanding, using nilpotence of $J$, $\| J\| = 1$, and submultiplicativity of the operator norm, we get
	\begin{equation} \label{eq:J-pseudo}
		|\lambda|^n \le \|(J + \delta E)^n\| \le (1 + \delta)^n - 1 = O(\delta)
	\end{equation}
	where the big-$O$ refers to the limit $\delta \to 0$ (recall $n$ is fixed).

	Writing $J + \delta E = V^{-1}DV$, we want to lower bound the condition number of $V$. As above, let $\lambda$ be an eigenvalue of $J + \delta E$, now writing $w^\ast$ and $v$ for its left and right eigenvectors. We'll use the lower bound
	$$
		\kappa(V) = \|V^{-1}\|\|V\| \ge  \frac{\|w^\ast\|\|v\|}{|w^\ast v|}=\kappa(\lambda).
	$$
	Since the formula above is agnostic to the scaling of the left and right eigenvectors, we'll assume that both have unit length and show that $|w^\ast v|$ is small.

	Let $0 \le k \le n$. Then $\Vert (J + \delta E)^k v \Vert = |\lambda|^k$, and analogously to \eqref{eq:J-pseudo}, 
	$$
	    \| (J + \delta E)^k - J^k\| \le (1 + \delta)^k - 1 = O(\delta).
	$$
	Since $J$ acts on the left as a left shift,
	\begin{align*} 
	    \left( \sum_{i=k+1}^n |v_i|^2 \right)^{1/2} 
	    &= \Vert J^k v \Vert\\
	    &\le \Vert (J+\delta E)^k v\Vert  + \Vert (J^k - (J+\delta E)^k)v\Vert \\
	    &\le |\lambda|^k + O(\delta) \\
	    &= O(\delta^{k/n}),
	\end{align*}
	where the final line follows from \eqref{eq:J-pseudo}. Similarly,
	\[ 
	    \left(\sum_{i=1}^{n-k} |w_i|^2 \right)^{1/2} = \Vert w^\ast J^k \Vert = O(\delta^{k/n}).
	\]
	Finally, we have $\kappa(V)^{-1} = |w^* v| \le \displaystyle\sum_{j=1}^n |w_j| |v_j|$, which in turn is at most
	\begin{align*}
	    \sum_{j=1}^n \left( \sum_{i=1}^j |w_i|^2 \right)^{1/2} \left( \sum_{i=j}^n |v_i|^2 \right)^{1/2} = O( \delta^{(n-j)/n} \delta^{(j-1)/n}) = O(\delta^{1 - 1/n}).
	\end{align*}

\end{proof}

We end by showing that the dependence on $n$ in Theorem \ref{thm:b} cannot be improved. 

\begin{proposition}There exists $c > 0$ such that for all $n$, 
$$
		\dE\sum_{i\in[n]}\kappa^2\left(\lambda_i(G_n)\right) \ge cn^2.
	$$
\end{proposition}

\begin{proof}
Bourgade and Dubach \cite[Theorem 1.1, Equation 1.8]{bourgade2018distribution} show that eigenvalue condition numbers in the bulk of the spectrum of complex Ginibre matrices are of order $\sqrt{n}$.  Precisely,
	$$
		\lim_{n \to \infty} \frac{\dE [ \kappa(\lambda_i)^2 | \lambda_i = z ] }{n} = 1 - |z|^2
	$$
	uniformly for (say) $z \in D(0,r)$ for any $r< 1$.  The classical \emph{circular law} for the limiting spectral distribution of Ginibre matrices ensures that 
	\[ \lim_{n \to \infty} \frac{ \dE | \Lambda(G_n) \cap D(0,r)|}{n} = \frac{\text{vol}(D(0, r))}{\text{vol}(D(0, 1))} = r^2 .\]  Thus,
	\[ \liminf_{n \to \infty}  \frac{\dE\sum_{i\in[n]}\kappa\left(\lambda_i(G_n)\right)^2}{n^2}  \ge r^2 (1 - r^2) > 0. \]
\end{proof}

\section{Conclusion and Discussion}\label{sec:conc}
A key theme in our work is the interplay between the related notions of eigenvector condition number $\kappa_V$, eigenvalue condition number $\kappa(\lambda_i)$ and pseudospectrum $\Lambda_\eps$. Equally important is the fact that global objects such as $\kappa_V$ and $\Lambda_\eps$ can be controlled by local quantities, specifically the least singular values of shifts $\sigma_n(zI - M)$ for each $z \in \mathbb{C}$. The proof also heavily exploits the left and right unitary invariance of the Ginibre ensemble (via Theorem \ref{thm:sniady}, due to \'Sniady) as well as anticoncentration of the complex Gaussian.

One natural question is whether similar results hold if one replaces Gaussian perturbations with a different class of random perturbations $G'$.  To apply the approach in this paper, the key difficulty would be obtaining suitable bounds for the least singular value of $z - A - \delta G'$. Davies \cite{davies} presents experimental evidence that Theorem \ref{thm:a} holds for random real rank-one perturbations and random real Gaussian perturbations, but a proof (or disproof) remains to be found.  See Remark \ref{rem:real} for a discussion of why our proof does not extend to the case of real Gaussian perturbations.

One may also ask if Theorem \ref{thm:a} can be derandomized; that is, if the regularizing perturbation $E$ can be chosen by a deterministic algorithm given $A$ as input. One natural choice would be to perturb in the direction of the nearest normal matrix in either operator or Frobenius norm, the latter of which can be written as a certain optimization problem over unitary matrices \cite{ruhe1987closest}.

Proposition \ref{prop:lowerdelta} shows that the upper bound in Theorem
\ref{thm:a} is tight in the perturbation size $\delta$.  Now, let $c_n$ be the
smallest constant such that Theorem \ref{thm:a} holds with an upper bound of
$c_n/\delta$.  Theorem \ref{thm:a} implies that $c_n \le 8 n^{3/2}$, and since
$\kappa_V  = \Vert V \Vert \Vert V^{-1} \Vert \ge 1$ for any matrix, we have
$c_n \ge 1$.  It would be interesting to determine the correct asymptotic
behavior of $c_n$. Davidson, Herrero, and Salinas asked in 1989 \cite{davidson1989quasidiagonal} whether the statement of
Theorem \ref{thm:a} is possible with $\kappa_V(A+E)$ depending only on $\delta$ and not on $n$.
In the present context, we can ask the more refined question: does Theorem
\ref{thm:a} hold with bounded $c_n$, or must $c_n$ go to infinity with $n$?

\section*{Acknowledgements}
J.B. is supported by the NSF Graduate Research Fellowship Program under Grant DGE-1752814. A.K., S.M., and N.S. are supported by NSF Grant CCF-1553751. This work was done at the Simons Institute at UC Berkeley as part of the Spring 2019 program "Geometry of Polynomials." We thank Jorge Garza Vargas for helpful conversations and for pointing out errors in an earlier version of this paper.  We thank Stanis\l{}aw Szarek for pointing out the reference \cite{davidson1989quasidiagonal} to us. We thank Amol Aggarwal, Milind Hegde, Tyler Maltba, and Jim Pitman for helpful conversations regarding stochastic differential equations. Finally, we thank the anonymous referee for thoughtful comments on an earlier version of this paper.

\bibliographystyle{alpha}

\bibliography{davies}

\appendix
\section{Proof of Theorem \ref{thm:sniadyreal}}
The goal of this appendix is to adapt \'Sniady's  \cite{sniady2002random} proof of Theorem \ref{thm:sniady}, as outlined below the statement of Theorem \ref{thm:sniady},  to the case of real matrices with real Ginibre perturbations.

The stochastic differential equation satisfied by the squared singular values of a real matrix Brownian motion was derived by Bru in her work on Wishart processes \cite{bru1989diffusions, bru1991wishart} and independently by Le in her work on shape theory \cite{le1994brownian, le1999singular}.  The equation reads as follows:

\begin{equation} \label{eqn:bru}
    d\lambda_i = \frac{2\sqrt{\lambda_i}}{n}\,dB_i + \left( 1+\sum_{j \ne i} \frac{\lambda_i + \lambda_j}{\lambda_i - \lambda_j}\right)\,dt, \qquad 1 \le i \le n.
\end{equation}

The proof strategy of \' Sniady crucially relies on the existence and uniqueness of strong solutions to the singular value SDE. This is needed in order to obtain two solutions driven by the same Brownian motion, and to assert that the law of each solution indeed matches the law of the singular values of a noncentered Ginibre matrix. See \cite{anderson2010introduction} for a definition of strong solution and a rigorous proof of existence and uniqueness of strong solutions for Dyson Brownian motion, the Hermitian analogue of the Ginibre singular values process.

Fortunately, such results are known for the SDE (\ref{eqn:bru}).  Let $\Lambda$ denote the domain
\[\Lambda \in \mathbb{R}^n := \{\lambda : 0 \le \lambda_n < \dots < \lambda_1\}.\]
For any initial data $\lambda(0)$ lying in the closure $\overline{\Lambda}$, it is known that strong solutions to (\ref{eqn:bru}) exist, are unique, and lie in $\Lambda$ for all $t > 0$, almost surely \cite[Corollary 6.5]{graczyk2014strong}.  Combining this with \cite[Theorem 1]{bru1989diffusions}, we have that for initial data $\lambda(0)$ lying in $\Lambda$, the law of the strong solutions to (\ref{eqn:bru}) matches the law of the squared singular values process of $A + M/\sqrt{n}$, where $M$ is a matrix of i.i.d. standard real Brownian motions and $A$ has squared singular values $\lambda(0)$.  (It should be possible to extend this last statement for initial data in $\overline{\Lambda}$, but the proof may be somewhat involved---cf. \cite{anderson2010introduction}, which contains a proof of the corresponding extension for Dyson Brownian motion.)  

Let $a_i(\lambda) = 1 + \sum_{j \ne i} \frac{\lambda_i + \lambda_j}{\lambda_i - \lambda_j}$ denote the drift coefficient in (\ref{eqn:bru}).  As in \'Sniady's proof for the complex Ginibre case (Theorem \ref{thm:sniady}), the key property of $a$ allowing for the comparison theorem is the so-called \emph{quasi-monotonicity} (see \cite{ding1998new}) or \emph{Kamke--Wa\.zewski condition} \cite[\S XI.13]{mitrinovic1991inequalities} from differential inequalities, which is simply that 
\begin{equation} \text{for all $i$, } a_i(\lambda^{(1)}) \le a_i(\lambda^{(2)})  \text{ whenever } \lambda^{(1)}_i = \lambda^{(2)}_i\text{ and } \lambda^{(1)}_j \le \lambda^{(2)}_j \text{ for all } j \ne i. \label{kw} \end{equation}  One easily checks that $a$ satisfies this condition on the domain $\Lambda$.

The nonconstant (indeed, non-Lipschitz) diffusion coefficient $2\sqrt{\lambda_i}/n$ in (\ref{eqn:bru}) is a technical obstacle which does not appear in the SDE (\ref{sde}) for the complex case.  Consequently, the final step of \'Sniady's proof as sketched below Theorem \ref{thm:sniady} cannot be repeated naively, because taking the difference of two solutions no longer cancels out the diffusion terms. Fortunately, theory has been developed to handle H\"older-$1/2$ diffusion coefficients; see \cite[\S IX.3]{revuz2013continuous} for exposition of the one-dimensional case and see \cite{krasin2010comparison} for a survey of comparison theorems for SDEs in general.

Quasi-monotonicity and the one-dimensional H\"older-$1/2$ comparison theory are combined in a rather general multidimensional comparison theorem of Gei{\ss} and Manthey \cite[Theorem 1.2]{geiss1994comparison}.  Applied to the SDE (\ref{eqn:bru}), this theorem provides exactly the right conclusion to replace the final step of \'Sniady's proof.  We state the relevant special case of their theorem below:
\begin{theorem}[Gei\ss-Manthey] \label{thm:geissmanthey}
Consider the SDE
\[dX_i = \sigma_i(X)\,dB_i + a_i(X)\,dt, \qquad 1 \le i \le n,\]
where the $B_i$ are independent standard real Brownian motions, and $\sigma_i, a_i : \mathbb{R}^n \to \mathbb{R}$ are continuous.  Suppose the following conditions are satisfied:
\begin{enumerate}
    \item the drift coefficient $a$ satisfies the quasi-monotonicity condition (\ref{kw})
    \item  there exists $\rho : \mathbb{R}_+ \to \mathbb{R}_+$ increasing with $\int_{0}^\eps \rho^{-2}(u)\,du = \infty$ for some $\eps>0$, such that $|\sigma_i(x) - \sigma_i(y)| \le \rho(|x_i - y_i|)$ for all $i$ and all $x, y \in \mathbb{R}^n$
    \item strong solutions for the SDE exist for all time and are unique.
\end{enumerate} Suppose initial conditions $X^{(1)}(0)$ and $X^{(2)}(0)$ satisfy the inequality $X^{(1)}_i(0) \le X^{(2)}_i(0)$ for all $i$.  Then almost surely, $X^{(1)}_i(t) \le X^{(2)}_i(t)$ for all $i$ and for all $t > 0$.  
\end{theorem}

Setting $\rho(u) := \sqrt{u}$, the SDE (\ref{eqn:bru}) satisfies the conditions of the Gei\ss-Manthey theorem, except that our domain for both $a_i$ and $\sigma_i$ is $\Lambda$, not $\mathbb{R}^n$.  We address these two coefficients in turn.

First we deal with the drift coefficient $a_i$, using a standard localization argument already implicit in the proof of Gei{\ss} and Manthey.  They (implicitly) define the stopping time $\vartheta_N$ to be the first time $\Vert X^{(1)} \Vert \ge N$ or $\Vert X^{(2)} \Vert \ge N$, and use the fact that $a$ is Lipschitz on the restricted domain $\Vert X \Vert \le N$ to show that
\[ \mathbb{P}\left[X^{(1)}_i(t) \le X^{(2)}_i(t)  \text{ for all } 0 \le t \le \vartheta_{N}\right] = 1.\] 
Since strong solutions exist for all time, we have $\vartheta_N \to \infty$ as $N \to \infty$ almost surely, which proves the theorem.  We modify this strategy for our SDE (\ref{eqn:bru}) in the standard way:  Define the stopping time $\tau_{1/m}$ to be the first time either $\lambda^{(1)}$ or $\lambda^{(2)}$ leaves the set 
\[\Lambda_{1/m} := \{\lambda \in \Lambda: |\lambda_i - \lambda_{i+1}| > 1/m \text{ for all }1 \le i \le n-1.\}.\]  Since strong solutions starting in $\Lambda$ stay in $\Lambda$ for all $t \ge 0$ and are continuous, we have $\tau_{1/m} \to \infty$ as $m \to \infty$ almost surely.  Since our $a$ is Lipschitz on $\Lambda_{1/m}$, the proof of Theorem \ref{thm:geissmanthey} shows that \[\mathbb{P}\left[\lambda^{(1)}_i(t) \le \lambda^{(2)}_i(t)  \text{ for all } 0 \le t \le \tau_{1/m}\right] = 1\] for all $m$. Taking $m \to \infty$, the result follows.

Finally, we address the diffusion coefficient $\sigma_i(\lambda) = 2\sqrt{\lambda_i}/n$.  The standard fix is to first modify the SDE to have diffusion coefficients $2\sqrt{|\lambda_i|}/n$ for all $i$, so that the domain of $\sigma_i$ is enlarged to $\mathbb{R}^n$ and Theorem \ref{thm:geissmanthey} may be applied. For this modified SDE, note that the constant zero function $\lambda^{(1)}(t) = 0$ is a strong solution.  Now let $\lambda^{(2)}$ be any solution with $\lambda^{(2)}_i(0) \ge 0$ for all $i$.  Applying Theorem \ref{thm:geissmanthey} to $\lambda^{(1)}$ and $\lambda^{(2)}$, we conclude that in fact, $\lambda^{(2)}(t) \ge 0$ for all $t \ge 0$.  Thus, the absolute value bars in the modified SDE can be removed a posteriori.  This argument is used, for example, when setting up the SDE for the so-called \emph{Bessel process}, which shares this square-root diffusion coefficient---see \cite[\S XI.1]{revuz2013continuous} for details.

\end{document}